\documentclass{article}
\usepackage{amsmath}
\usepackage{amsfonts}
\usepackage{amsthm}
\usepackage{centernot}
\usepackage{url}

\usepackage{enumitem}

\def\natN{  \mathbb{N} }

\newtheorem{theorem}{Theorem}[section]
\newtheorem{lemma}{Lemma}[section]

\title{An application of Baker's method to the Je\'smanowicz' conjecture on primitive Pythagorean triples}
\author{Maohua Le }
\date{17 October 2018}

\begin{document}

\maketitle

\begin{abstract}

Let $m$, $n$ be positive integers such that $m>n$, $\gcd(m,n)=1$ and $m \not\equiv n \bmod 2$.  In 1956, L. Je\'smanowicz \cite{Jes} conjectured that the equation $(m^2 - n^2)^x + (2mn)^y = (m^2+n^2)^z$ has only the positive integer solution $(x,y,z) = (2,2,2)$.  This problem is not yet solved.  
In this paper, combining a lower bound for linear forms in two logarithms due to M. Laurent \cite{Lau} with some elementary methods, we prove that if $mn \equiv 2 \bmod 4$ and $m > 30.8 n$, 
then Je\'smanowicz' conjecture is true.     
\end{abstract}

\bigskip



\section{Introduction}  

Let $\natN$ be the set of positive integers.  Let $m$, $n$ be positive integers such that $m>n$, $\gcd(m,n)=1$ and $m \not\equiv n \bmod 2$.  It is well known that the triple $(m^2-n^2, 2mn, m^2+n^2)$ is a primitive Pythagorean triple with $(m^2 - n^2)^2 + (2mn)^2 = (m^2+n^2)^2$.  In 1956 L. Je\'smanowicz \cite{Jes} conjectured that the equation 
$$(m^2 - n^2)^x + (2mn)^y = (m^2+n^2)^z, x,y,z \in \natN  \eqno{(1.1)} $$ 
has only the solution $(x,y,z) = (2,2,2)$.  This problem is not solved yet.  

For over twenty years, many papers have investigated Je\'smanowicz' conjecture for the case that $mn \equiv 2 \bmod 4$.  In this respect, Je\'smanowicz' conjecture is true in the following cases:
\begin{enumerate}[label=(\roman*)]
  \item (M.-H. Le \cite{Le8})  $mn \equiv 2 \bmod 4$ and $m^2+n^2$ is an odd prime power. 
  \item (Z.-F. Cao \cite{Cao})  $(m,n) \equiv (5,2) \bmod 8$.
  \item (M.-J. Deng and D.-M. Huang \cite{DeHu}) $mn \equiv (2,3) \bmod 4$ and either $m+n \not\equiv 1 \bmod 16$ or $y>1$.  
  \item (M.-M. Ma and Y.-G. Chen \cite{MaCh})  $mn \equiv 2 \bmod 4$ and $y>1$.
  \item (K. Takakuwa and Y. Asaeda \cite{TaAs}) $m \equiv 2 \bmod 4$, $n=3$ and $m$ satisfies some conditions.  
  \item (Y.-D. Guo and M.-H. Le \cite{GuLe}) $m \equiv 2 \bmod 4$, $n=3$ and $m>6000$.  
  \item (K. Takakuwa \cite{Tak}) $m \equiv 2 \bmod 4$ and $n \in \{ 3,7,11,15 \}$. 
  \item (N. Terai \cite{Ter}) $n=2$.
  \item (M.-J. Deng and J. Guo \cite{DeGu}) $n\equiv 2 \bmod 4$ and $n < 600$.  
  \item (M.-H. Le \cite{Le9}) $(m,n) \equiv (2,3) \bmod 4$ and $m > 81 n$.  
  \item (T. Miyazaki and N. Terai \cite{MiTe})  $n \equiv 2 \bmod 4$, $m>72n$ and the divisors of $n$ satisfy some conditions. 
  \item (P.-Z. Yuan and Q. Han \cite{YuHa}) $mn \equiv 2 \bmod 4$, $m>72n$ and the divisors of $m$, $n$ satisfy some conditions.  
\end{enumerate} 

In this paper, combining a lower bound for linear forms in two logarithms due to M. Laurent \cite{Lau} with some elementary methods, we improve the results of \cite{Le9}, \cite{MiTe} and \cite{YuHa} as follows:

\begin{theorem} 
If $mn \equiv 2 \bmod 4$ and $m > 30.8 n$, then Je\'smanowicz' conjecture is true.  
\end{theorem}

\section{Preliminaries}  

\begin{lemma}[\cite{DeGu}] 
If $n \equiv 2 \bmod 4$ and $n<600$ then Je\'smanowicz' conjecture is true.
\end{lemma}

\begin{lemma}[Corollary 1.1 of \cite{YuHa}]  
If $m \equiv 2 \bmod 4$, 
$n \not\equiv 1 \bmod 8$ and $n<85$, then Je\'smanowicz' conjecture is true.
\end{lemma}

\begin{lemma}[\cite{HaYu}]  
If $mn \equiv 2 \bmod 4$ and $m+n$ has a prime divisor $p$ with $p \not\equiv 1 \bmod 16$, then Je\'smanowicz' conjecture is true.
\end{lemma} 

\begin{lemma} 
If $mn \equiv 2 \bmod 4$, $m>30.8 n$ and (1.1) has a solution $(x,y,z) \ne (2,2,2)$, then $m^2+n^2 > m^2-n^2 \ge  2704$.  
\end{lemma} 

\begin{proof}
By Lemma 2.3, if $mn \equiv 2 \bmod 4$ and (1.1) has a solution $(x,y,z)$ with $(x,y,z) \ne (2,2,2)$, then every prime divisor $p$ of $m+n$ satisfies $p \equiv 1 \bmod 16$.  This implies that $m+n \equiv 1 \bmod 16$.  Hence, we have $m+n \equiv 1 \bmod 4$ and $(m,n) \equiv (2,3)$ or $(3,2) \bmod 4$.  Therefore, if $m \equiv 2 \bmod 4$, then $n \equiv 3 \bmod 4$ and 
$$ n \ge 85.  \eqno{(2.1)} $$
Similarly, by Lemma 2.1, (2.1) holds if $n \equiv 2 \bmod 4$.  Thus, since $m > 30.8 n$, we get from (2.1) that 
$$ m^2 + n^2 > m^2 - n^2 \ge m+n > 31.8 n \ge 31.8 \times 85 = 2703.  \eqno{(2.2)} $$
\end{proof}

\begin{lemma}[\cite{Ter}]  
If $m^2+n^2 \equiv 1 \bmod 2mn$, then Je\'smanowicz' conjecture is true.
\end{lemma}

\begin{lemma}[\cite{MaCh}] 
If $mn \equiv 2 \bmod 4$ and (1.1) has a solution $(x,y,z)$ with $(x,y,z)\ne (2,2,2)$, then the solution satisfies $x \equiv 0 \bmod 2$, $y=1$ and $z \equiv 1 \bmod 2$.  
\end{lemma}

\begin{lemma} 
Let $a_1$, $a_2$, $b_1$, $b_2$ be positive integers such that $\min\{a_1, a_2\} > 1$ and $\gcd(a_1, a_2)=1$.  Further let $\Lambda = b_1 \log(a_1) - b_2 \log(a_2)$.  Let $\rho$ and $\mu$ be real numbers with $\rho >1$ and $1/3 \le \mu \le 1$. Further let 
$$ \delta = \frac{1}{2}(1+2\mu - \mu^2), \lambda = \delta \log(\rho).  \eqno{(2.3)} $$
Then 
$$ \log|\Lambda| \ge  - C A_1 A_2 B^2 - \sqrt{\omega \theta} B - \log(C' A_1 A_2 B^2), \eqno{(2.4)} $$
where
$$A_j \ge \max\{ 1, (\rho+1) \log(a_j) \}, A_1 A_2 \ge \lambda^2, j= 1,2, \eqno{(2.5)} $$
$$ B \ge \log(\rho) + \max\left\{ \frac{\log(2)}{2} , \lambda, 1.81 + \log(\lambda) + \log\left( \frac{b_1}{A_2} + \frac{b_2}{A_1} \right)\right\}, \eqno{(2.6)} $$
$$ \omega = 2 + 2 \left( 1 + \frac{1}{4H^2} \right)^{1/2} , \theta  = \frac{1}{2H} + \left( 1 + \frac{1}{4H^2} \right)^{1/2} , H = \frac{B}{\lambda}, \eqno{(2.7)} $$
$$ C = \frac{C_0 \mu }{\lambda^3 \delta}, C' = \frac{\sqrt{C_0 \omega \theta}}{\lambda^3}, \eqno{(2.8)} $$
$$C_0 = \left( \frac{\omega}{6} + \frac{1}{2} \left( \frac{\omega^2}{9} + \frac{ 8 \lambda \omega^{5/4} \theta^{1/4} }{3 \sqrt{A_1A_2 H}} + \frac{4}{3} \left(\frac{1}{A_1} + \frac{1}{A_2} \right) \frac{\lambda \omega}{H} \right)^{1/2} \right)^2
.  \eqno{(2.9)} $$
\end{lemma}

\begin{proof}  
This lemma is the special case of Theorem 2 of \cite{Lau} for $\gamma_1$ and $\gamma_2$ coprime positive integers.
\end{proof}

\begin{lemma}  
Under the assumptions of Lemma 2.7, if $\min\{ a_1, a_2 \} \ge 2704$ and $ b_1/A_2 > b_2 / A_1 > 240$, then
$$ \log| \Lambda | > -14.8365 \log(a_1) \log(a_2) \left( 1.8248 + \log\left( \frac{b_1}{\log(a_2)} + \frac{b_2}{\log(a_1)} \right) \right)^2.   \eqno{(2,.10)}  $$
\end{lemma}  

\begin{proof}  
By Lemma 2.7, we may choose parameters 
$$ \rho = e^{1.575}, \mu = \frac{1}{3}.  \eqno{(2.11)}  $$
By (2.3) and (2.11), we have 
$$ \delta = \frac{7}{9}, \lambda = 1.225.  \eqno{(2.12)} $$
Since $\min\{a_1, a_2 \} > 2703$, by (2.5), (2.11) and (2.12), we can set 
$$ A_j = 5.8314 \log(a_j), j = 1,2.  \eqno{(2.13)} $$
Hence, by (2.13), we have 
$$A_j > 46.0803, j = 1,2. \eqno{(2.14)} $$
By (2.6), (2.11) and (2.12), we can set 
$$ B = 3.5880 + \log\left( \frac{b_1}{A_2} + \frac{b_2}{A_1} \right).  \eqno{(2.15)} $$
Since $b_1/A_2 > b_2/A_1 > 240$, by (2.15) and (2.7), we have 
$$ B > 9.7617 \eqno{(2.16)} $$
and 
$$ H > 7.9688.  \eqno{(2.17)} $$
Therefore, by (2.7) and (2.17), we get 
$$ \omega < 4.0040, \theta < 1.0648.  \eqno{(2.18)} $$
Further, by (2.9), (2.12), (2.14), (2.17) and (2.18), we have 
$$ C_0 < 1.8706, \eqno{(2.19)} $$
and by (2.8), we get
$$ C < 0.4361, C' < 2.0829.  \eqno{(2.20)} $$

From (2.4) and (2.20), we have 
$$ \log|\Lambda| > -\left( 0.4361 + \frac{ \sqrt{ \omega \theta }}{A_1 A_2 B} + \frac{ \log(2.0829 A_1 A_2 B^2)}{A_1 A_2 B^2} \right) A_1 A_2 B^2.  \eqno{(2.21)} $$
By (2.14), (2.16) and (2.18), we get 
$$ \frac{ \sqrt{ \omega \theta }}{A_1 A_2 B} < 1.0026 \times 10^{-4}.  \eqno{(2.22)}  $$
Since $f(t) = \log(t) /  t$ is a decreasing function for $t>e$, by (2.14) and (2.16), we have 
$$  \frac{ \log(2.0829 A_1 A_2 B^2)}{A_1 A_2 B^2} < \frac{ \log(2.0829 \times 46.0803^2 \times 9.7617^2 ) }{ 46.0803^2 \times 9.7617^2} < 0.6401 \times 10^{-4}.  \eqno{(2.23)} $$
Thus, by (2.13), (2.15), (2.21), (2.22) and (2.23), we obtain
$$ \log|\Lambda| > -0.4363  A_1 A_2 B^2 > -14.8365 \log(a_1) \log(a_2) \left( 1.8248 + \log\left( \frac{b_1}{\log(a_2)} + \frac{b_2}{\log(a_1)} \right) \right)^2.  \eqno{(2.24)} $$
The lemma is proved.  
\end{proof}

\section{Proof of Theorem} 

We mow assume that $mn \equiv 2 \bmod 4$, $m>30.8 n$ and (1.1) has a solution $(x,y,z)$ with $(x,y,z) \ne (2,2,2)$.  By Lemma 2.6, the solution $(x,y,z)$ satisfies 
$$ (m^2-n^2)^x + 2mn = (m^2+n^2)^z, x,z \in \natN, x \equiv 0 \bmod 2, z \equiv 1 \bmod 2.  \eqno{(3.1)} $$
By Lemma 2.4, $m^2+n^2$ and $m^2-n^2$ satisfy (2.2).  Since $x \equiv 0 \bmod 2$, if $x\le z$ then from (3.1) we get 
$ 2mn = (m^2 + n^2)^z-(m^2-n^2)^x \ge (m^2 + n^2)^x-(m^2-n^2)^x \ge (m^2+n^2)^2 - (m^2-n^2)^2 = (2mn)^2 > 2mn,$
a contradiction.  Therefore, we have 
$$ x > z \eqno{(3.2)} $$
and $x-z$ an odd integer.  

Since $x \equiv 0 \bmod 2$ and $(m^2-n^2)^2 \equiv (m^2+n^2)^2 \bmod (2mn)$, we have $(m^2-n^2)^x \equiv (m^2+n^2)^x \bmod (2mn)$.  Further, since $\gcd(2mn, m^2+n^2)=1$, by (3.1) and (3.2), we get
$$(m^2+n^2)^{x-z} \equiv 1 \bmod (2mn).  \eqno{(3.3)} $$
Hence, by Lemma 2.5, we see from (3.3) that the case $x-z=1$ can be removed.  So we have 
$$ x-z \ge 3.  \eqno{(3.4)}$$
Let $k = m/n$.  By (3.1) and (3.4), we have 
$$ (m^2-n^2)^3 \le (m^2-n^2)^{x-z} < \left( \frac{ m^2+n^2}{m^2-n^2} \right)^z = \left( 1+ \frac{2}{k^2-1} \right)^z.  \eqno{(3.5)} $$
Notice that $\log(1+t) < t$ for any $t>0$.  By (3.5), we get 
$$ 3 \log(m^2-n^2) < \frac{2z}{k^2-1}.  \eqno{(3.6)} $$
Since $k > 30.8$, we see from (3.6) that 
$$ \frac{z}{\log(m^2-n^2)} > \frac{3}{2} (k^2-1) > 1421.46.  \eqno{(3.7)}$$

On the other hand, by (3.1), we have 
$$ 
\begin{aligned} 
z \log(m^2+n^2) =& \log( (m^2-n^2)^x + 2mn) =   \log((m^2-n^2)^x) + \\  
&\log\left( 1 + \frac{2mn}{(m^2-n^2)^x} \right) < x \log(m^2-n^2) + \frac{2mn}{(m^2-n^2)^x}.
\end{aligned}
 \eqno{(3.8)}  
$$
Let $(a_1, a_2, b_1, b_2)= ( m^2+n^2, m^2-n^2, z, x)$ and $\Lambda = z \log(m^2+n^2) - x \log(m^2-n^2)$.  By (3.8), we have 
$$ 0 < \Lambda < \frac{2mn}{(m^2-n^2)^x}.  \eqno{(3.9)} $$
Further, since $(m^2-n^2)^x > (m^2-n^2)^2 \ge (m+n)^2 > 2mn$, we see from (3.1) that $2(m^2-n^2)^x > (m^2+n^2)^z$.  Hence, by (3.9), we get 
$$ 0 < \Lambda < \frac{4mn}{(m^2+n^2)^z},  \eqno{(3.10)} $$
whence we obtain
$$ \log(4mn) - \log|\Lambda| > z \log(m^2+n^2).  \eqno{(3.11)} $$
Since $x \ge 4$, by (2.2), (3.7) and (3.8), we have
\begin{equation*}
\begin{aligned} 
 \frac{z}{\log(m^2-n^2)} > \frac{x}{\log(m^2+n^2)} > \frac{z}{\log(m^2-n^2)} - \frac{2mn}{ (m^2-n^2)^x \log(m^2+n^2)} \frac{1}{\log(m^2-n^2)} \\
> \frac{z}{\log(m^2-n^2)} - \frac{1}{ (m+n)^2 (\log(m+n))^2} >  1421.46 - 5 \times 10^{-8}, 
\end{aligned}
\eqno{(3.12)} 
\end{equation*}
whence we get 
$$ \frac{z}{5.8314 \log(m^2-n^2)} > \frac{x}{5.8314 \log(m^2+n^2)} > 240.  \eqno{(3.13)} $$
Therefore, by Lemma 2.8, we see from (2.2) and (3.13) that 
$$ \log|\Lambda| >  -14.8365  \log(m^2+n^2) \log(m^2-n^2) \left( 1.8248 + \log\left( \frac{z}{\log(m^2-n^2)} + \frac{x}{\log(m^2+n^2)} \right) \right)^2.  \eqno{(3.14)} $$
Substituting (3.14) into (3.11), by (3.13), we have 
$$ \frac{ \log(4mn)}{ \log(m^2+n^2) \log(m^2-n^2)} + 14.8365 \left( 1.8248 + \log\left(\frac{2z}{\log(m^2-n^2)} \right) \right)^2 > \frac{z}{\log(m^2-n^2)}. \eqno{(3.15)} $$
Since $m^2+n^2 > 2mn$ and $m^2-n^2 \ge m+n$, by (2.2), we have
$$ \frac{\log(4mn)}{\log(m^2+n^2) \log(m^2-n^2)} < \frac{1}{\log(m+n)} \left( \frac{\log(2)}{\log(m+n)} + 1 \right) < 0.1266.  \eqno{(3.16)} $$
Hence, by (3.15) and (3.16), we get 
$$ 0.1266 + 14.8365 \left(1.8248 + \log\left( \frac{2z}{\log(m^2-n^2)} \right) \right)^2 > \frac{z}{\log(m^2-n^2)}.  \eqno{(3.17)} $$

Let $f(t) = t - 0.1266 - 14.8365 (1.8248 + \log(2t))^2$.  We see from (3.17) that 
$$ f\left( \frac{z}{\log(m^2-n^2)} \right) <0.  \eqno{(3.18)} $$
Since $f'(t) = 1 - 29.6730 (1.8248 + \log(2t))/t$ and $f'(t)>0$ for $t>250$, where $f'(t)$ is the derivative of $f(t)$, we have $f(t) > f(1420) > 0$ for $t>1420$.  Therefore, by (3.18), we get $z / \log(m^2-n^2) < 1420$, which contradicts (3.7).  Thus, if $mn \equiv 2 \bmod 4$ and $m>30.8 n$, then (1.1) has only the solution $(x,y,z)=(2,2,2)$.  The theorem is proved.

\bigskip

\noindent
Acknowledgment: The author would like to thank Prof. R. Scott and Prof. R. Styer for reading the original manuscript carefully and giving valuable advice.  Especially thanks to Prof. R. Styer for checking the calculations and providing other technical assistance.

\bigskip 

Maohua Le

Institute of Mathematics

Lingnan Normal College

Zhangjiang, Guangdong  524048

China

\end{document}